\documentclass[12pt]{amsart}

\usepackage{amsmath}%
\usepackage{amsfonts}%
\usepackage{amssymb}%
\usepackage{graphicx}
\usepackage{setspace}
\usepackage{enumerate}
\usepackage{mathrsfs}

\title[Sobolev extensions of Lipschitz mappings]{Sobolev extensions of Lipschitz mappings into metric spaces}
\author{Scott Zimmerman}

\address{S.\ Zimmerman: Department of Mathematics, University of Connecticut, 
	341 Mansfield Road U1009, Storrs, Connecticut 06269, USA, {\tt srz5@pitt.edu}}

\thanks{This work was supported by the NSF grant DMS-1500647 of Piotr Haj\l{}asz.}

plus 6pt minus 12pt
\newtheorem{theorem}{Theorem}
\newtheorem{lemma}[theorem]{Lemma}
\newtheorem{corollary}[theorem]{Corollary}
\newtheorem{proposition}[theorem]{Proposition}

\def\diam{{\rm diam\,}}

\theoremstyle{definition}

\newtheorem{definition}[theorem]{Definition}

\newcommand{\barint}{
\rule[.036in]{.12in}{.009in}\kern-.16in \displaystyle\int }

\newcommand{\barcal}{\mbox{$ \rule[.036in]{.11in}{.007in}\kern-.128in\int $}}

\def\diam{\operatorname{diam}}

\def\mvint_#1{\mathchoice
          {\mathop{\vrule width 6pt height 3 pt depth -2.5pt
                  \kern -8pt \intop}\nolimits_{\kern -3pt #1}}%
          {\mathop{\vrule width 5pt height 3 pt depth -2.6pt
                  \kern -6pt \intop}\nolimits_{#1}}%
          {\mathop{\vrule width 5pt height 3 pt depth -2.6pt
                  \kern -6pt \intop}\nolimits_{#1}}%
          {\mathop{\vrule width 5pt height 3 pt depth -2.6pt
                  \kern -6pt \intop}\nolimits_{#1}}}

\numberwithin{theorem}{section} \numberwithin{equation}{section}

\begin{document}

\begin{abstract}
Wenger and Young proved that the pair $(\mathbb{R}^m,\mathbb{H}^n)$ has the Lipschitz extension property for $m \leq n$
where $\mathbb{H}^n$ is the sub-Riemannian Heisenberg group. 
That is, for some $C>0$, any $L$-Lipschitz map from a subset of $\mathbb{R}^m$ into $\mathbb{H}^n$
can be extended to a $CL$-Lipschitz mapping on $\mathbb{R}^m$.
In this paper, we construct Sobolev extensions of such Lipschitz mappings with no restriction on the dimension $m$.
We prove that any Lipschitz mapping from a compact subset of $\mathbb{R}^m$ into $\mathbb{H}^n$
may be extended to a Sobolev mapping on any bounded domain containing the set.
More generally, we prove this result in the case of mappings into any Lipschitz $(n-1)$-connected metric space.
\end{abstract}

\maketitle

\section{Introduction}

A pair of metric spaces $(X,Y)$ has the \emph{Lipschitz extension property} 
if there is a constant $C>0$ so that any $L$-Lipschitz mapping $f:A \to Y$, $A \subset X$ 
has a $CL$-Lipschitz extension $F:X \to Y$.
Recall that a mapping $f:X\to Y$ between metric spaces is $L$-Lipschitz for some $L>0$ 
if $d_Y(f(x),f(y)) \leq L d_X(x,y)$ for all $x,y \in X$.
Extensive research has been conducted in the area of Lipschitz extensions.
See, for example, \cite{brudnyi,fassler,extend_banach,lang_sch,lang_schr,magnani,extend_book,wenger_young}.
Wenger and Young \cite{wenger_young} showed that $(\mathbb{R}^m,\mathbb{H}^n)$ 
has the Lipschitz extension property for $m \leq n$
where $\mathbb{H}^n$ is the sub-Riemannian Heisenberg group of topological dimension $2n+1$ (see Section \ref{sec_heis} for definitions).
More generally, the authors proved that $(X,\mathbb{H}^n)$ 
has the Lipschitz extension property
as long as the Assouad-Nagata dimension of $X$ is at most $n$ (see \cite{assouad,lang_sch,wenger_young}).
For such metric spaces $X$, Lang and Schlichenmaier \cite{lang_sch} showed that,
when $Y$ is any Lipschitz $(n-1)$-connected metric space,
there is a constant $C>0$ so that any $L$-Lipschitz mapping $f:A \to Y$ 
defined on a closed subset $A \subset X$ 
has a $CL$-Lipschitz extension $F:X \to Y$.
A metric space $Y$ is 
\emph{Lipschitz} $(n-1)$-\emph{connected} if there is a constant $\gamma \geq 1$ so that 
any $L$-Lipschitz map $f:S^k \to Y$ ($L\geq 0$) on the $k$-dimensional sphere has a 
$\gamma L$-Lipschitz extension $F:B^{k+1} \to Y$ on the $(k+1)$-ball for $k =0,1,\dots,n-1$. 
The result of Wenger and Young follows immediately if one proves the Lipschitz $(n-1)$-connectivity of $\mathbb{H}^n$.
As Wenger and Young mentioned, however, proving this property for $\mathbb{H}^n$ is difficult,
and thus they provided a direct proof of their Lipschitz extension result.
As a consequence, the metric space $\mathbb{H}^n$ is indeed Lipschitz $(n-1)$-connected.

What happens, however, when the dimension of the domain is large?
As Balogh and F\"{a}ssler \cite{rekt} showed,
the pair $(\mathbb{R}^m,\mathbb{H}^n)$
does not have the Lipschitz extension property when $m > n$.
Indeed, there is a bi-Lipschitz embedding of the sphere $S^n$ into $\mathbb{H}^n$,
and one can show that this embedding does not admit a Lipschitz extension to the ball $B^{n+1}$.
Since $B^{n+1}$ can be regarded as a subset of $\mathbb{R}^m$ for any $m > n$, 
the result follows.
(See also Theorems 1.5 and 1.6 in \cite{haj_spheres} for a shorter proof.)
This result was extended to include mappings into and between jet space Carnot groups 
in \cite{jet_space}.

In this paper, we consider Sobolev extensions of Lipschitz mappings $f:A \to \mathbb{H}^n$, $A \subset \mathbb{R}^m$.
Since Sobolev mappings form a larger class than Lipschitz mappings,
it turns out that, in the Sobolev case, we no longer have any restriction on the dimension of the domain.
The first main result of the paper is stated here.
Throughout the paper, a \emph{domain} $\Omega$ in $\mathbb{R}^m$ will be an open, connected set $\Omega \subset \mathbb{R}^m$.

\begin{theorem}
\label{main}
Fix $m,n \in \mathbb{N}$.
Suppose $Z \subset \mathbb{R}^m$ is compact 
and $\Omega$ is a bounded domain in $\mathbb{R}^m$ 
with $Z \subset \Omega$.
For $1 \leq p < n+1$ and any $L$-Lipschitz mapping $f:Z \to \mathbb{H}^n$, $L \geq 0$,
there exists a mapping
$F:\Omega \to \mathbb{H}^n$ in the class 
$W^{1,p}(\Omega,\mathbb{H}^n)$ 
such that $F(x) = f(x)$ for all $x \in Z$.

\noindent Moreover, there is a constant $C>0$ depending only on $m$, $n$, and $p$ such that,
if we write $F=(F_1,\dots,F_{2n},F_{2n+1})$, 
then $\Vert\partial F_j / \partial x_k\Vert_{L^p(\Omega)} \leq CL \left( \diam(\Omega) \right)^{m/p}$ for $k = 1,\dots,m$ and $j=1,\dots,2n$.
\end{theorem}

We will also see from the construction that the extension $F$
is in the class $ACL^p(\Omega,\mathbb{R}^{2n+1})$.
(See Section~\ref{sec_sob} for the appropriate definitions.)
Note that the bounds in Theorem~\ref{main} are given only for $j < 2n+1$. 
Such a condition follows naturally from the sub-Riemannian geometry of the Heisenberg group.
A brief explanation of this follows Definition~\ref{definition} in Section~\ref{sec_sob}.

If $m \leq n$, then $f$ admits a Lipschitz extension 
since $\mathbb{H}^n$ is Lipschitz $(n-1)$-connected
by the result of Wenger and Young, 
and this extension belongs to $W^{1,p}(\Omega,\mathbb{H}^n)$ 
for $1 \leq p \leq \infty$.
However, if $m > n$ we have the following possibility:
\begin{proposition}
\label{sharp}
There is a Lipschitz mapping $f:S^n \to \mathbb{H}^n$ 
such that there is no mapping $F:B^{n+1} \to \mathbb{H}^n$
satisfying both
$F \in W^{1,n+1}(B^{n+1},\mathbb{H}^n)$ and
$F|_{S^n}=f$.
\end{proposition}

Here, the restriction $F|_{S^n}=f$ is understood in the sense of traces
in $W^{1,n+1}(B^{n+1},\mathbb{R}^{2n+1})$.
That is, if $\{F_k\}$ is a sequence of mappings 
$F_k:B^{n+1} \to \mathbb{R}^{2n+1}$
which are $C^1$ up to the boundary 
and such that $F_k \to F$ in the norm of $W^{1,n+1}(B^{n+1},\mathbb{R}^{2n+1})$,
then $F_k|_{S^n}$ (classical function restriction) converges to $f$ in $L^p(S^n)$.

One such mapping $f:S^n \to \mathbb{H}^n$ is the bi-Lipschitz embedding used by Balogh and F\"{a}ssler \cite{rekt}.
In the proof of Proposition~\ref{sharp}, we will see ideas from
\cite[Theorem 2]{haj_geo},
\cite[Theorem 2.3]{haj_sob},
and \cite[Theorem 1.5]{haj_spheres}.

For mappings with Euclidean target,
Sobolev extension results like Theorem \ref{main} provide extensions defined on all of $\mathbb{R}^m$
via multiplication by a cutoff function.
However, since multiplication by a cutoff function in $\mathbb{R}^{2n+1} = \mathbb{H}^n$
does not necessarily preserve the weak contact equation \eqref{weak_contact}, 
such a simple argument will not work here.
We instead have the following local Sobolev extension defined on all of $\mathbb{R}^m$.
\begin{corollary}
Fix $m,n \in \mathbb{N}$.
Suppose $Z \subset \mathbb{R}^m$ is compact.
For $1 \leq p < n+1$ and any $L$-Lipschitz mapping $f:Z \to \mathbb{H}^n$, $L \geq 0$,
there exists a mapping
$\tilde{F}:\mathbb{R}^m \to \mathbb{H}^n$ in the class 
$W^{1,p}_{loc}(\mathbb{R}^m,\mathbb{H}^n)$ 
such that $F(x) = f(x)$ for all $x \in Z$.
\end{corollary}
This follows easily from the theorem.
Indeed, suppose $\Omega$ is a cube containing $Z$
and $\Phi:\mathbb{R}^m \to \Omega$ is a diffeomorphism which fixes $Z$.
Then, if $F \in W^{1,p}(\Omega,\mathbb{H}^n)$ is the extension from Theorem~\ref{main},
it follows that $\tilde{F} := F \circ \Phi \in W^{1,p}_{loc}(\mathbb{R}^m,\mathbb{H}^n)$.

It follows from classical Lipschitz extension proofs that
there is a constant $C>0$ so that any $L$-Lipschitz mapping $f:A \to Y$ 
defined on a closed subset $A \subset \mathbb{R}^m$ 
has a $CL$-Lipschitz extension $F:X \to Y$
when $Y$ is any Lipschitz $(n-1)$-connected metric space and $m \leq n$
(see \cite{almgren,extend_banach} or the proof of Lemma~\ref{wengeryoung}).
It turns out that Theorem~\ref{main} can be generalized to the case when the target space $\mathbb{H}^n$
is replaced by an arbitrary Lipschitz $(n-1)$-connected metric space $Y$.
In this case, our extension will be in the \emph{Ambrosio-Reshetnyak-Sobolev class} $AR^{1,p}(\Omega,Y)$.
For a bounded domain $\Omega$ in $\mathbb{R}^m$ and $1 \leq p < \infty$,
a mapping $F: \Omega \to Y$ belongs to the class $AR^{1,p}(\Omega,Y)$
if there is a non-negative function $g \in L^p(\Omega)$ satisfying the following:
for any $K$-Lipschitz $\phi:Y \to \mathbb{R}$, 
we have $\phi \circ F \in W^{1,p}(\Omega)$ 
and $|\partial (\phi \circ F) / \partial x_k(x)| \leq K g(x)$ 
for $k=1,\dots,m$ and almost every $x \in \Omega$.
This class of mappings was first introduced in \cite{ambrosio} and \cite{reshet1}.

\begin{theorem}
\label{main2}
Fix $m,n \in \mathbb{N}$.
Suppose $Z \subset \mathbb{R}^m$ is compact, 
$\Omega$ is a bounded domain in $\mathbb{R}^m$ 
with $Z \subset \Omega$,
and $Y$ is a Lipschitz $(n-1)$-connected metric space with constant $\gamma$.
For $1 \leq p < n+1$ and any $L$-Lipschitz mapping $f:Z \to Y$, $L \geq 0$,
there exists a mapping
$F:\Omega \to Y$ in the class 
$AR^{1,p}(\Omega,Y)$ 
such that $F(x) = f(x)$ for all $x \in Z$.
In particular, $\phi \circ F \in ACL^p(\Omega)$
for any Lipschitz $\phi: Y \to \mathbb{R}$.

\noindent Moreover, there is a constant $C>0$ depending only on $m$, $n$, $p$, and $\gamma$
such that we may choose $g \in L^p(\Omega)$ in the definition of $AR^{1,p}(\Omega,Y)$ 
with $\Vert g\Vert _{L^p(\Omega)} \leq CL \left( \diam(\Omega) \right)^{m/p}$.
\end{theorem}

Notice that, as before, there is no restriction on the dimension of the domain.
The theory of Sobolev mappings into metric spaces has been studied extensively in \cite{ambrosio,harmonic,peano,quasicon,shanmugalingam,sobolev_book,kor_sch,reshet1,reshet2}.
In particular, $\mathbb{H}^n$ valued Sobolev mappings have been explored in \cite{haj_weak,cap_lin,haj_lack,peano,mag_mal_mon}.
One motivation for the study of Sobolev extensions 
stems from the problem of approximating Sobolev mappings by Lipschitz ones \cite{bethuel,bos_pon_sch,haj_lack,haj_density,haj_spheres,hang_lin}.
In fact, the proof of Theorem~\ref{main2} employs the so called zero degree homogenization discussed in \cite{bos_pon_sch,haj_density}.

As we will see in Proposition \ref{heisenberg_sobolev},
$W^{1,p}(\Omega,\mathbb{H}^n)$ is contained in $AR^{1,p}(\Omega,\mathbb{H}^n)$.
Furthermore, in the case of bounded mappings, the two definitions of the Sobolev class are equivalent.
Hence Theorem~\ref{main} will be proven as a corollary to Theorem~\ref{main2}.

The format of the paper is as follows.
In Section \ref{sec_heis}, the Heisenberg group $\mathbb{H}^n$ is defined and relevant geometric properties are introduced.
The topic of Sobolev mappings into metric spaces, and in particular into $\mathbb{H}^n$, is addressed in Section~\ref{sec_sob}, 
and the section ends with the proof of Proposition~\ref{sharp}.
Section \ref{AN} introduces the Whitney triangulation of an open set in $\mathbb{R}^m$ and establishes a Lipschitz extension lemma. 
These are the primary tools used in the proof of Theorem~\ref{main2} given in Section~\ref{sec_main}.
The short proof of Theorem~\ref{main} then follows.

The author would like to extend thanks to his advisor Piotr Haj\l{}asz 
for his assistance in discovering this problem and many helpful conversations about the solution
and to the referees for their very helpful recommendations
which led to an improvement of the paper.

\section{The Heisenberg Group}
\label{sec_heis}

The Heisenberg group $\mathbb{H}^n$ is $\mathbb{R}^{2n+1}$ given the structure of a Lie group with multiplication
\begin{align*}
(x_1,y_1, \dots, x_n,y_n,t)&*(x_1',y_1', \dots, x_n',y_n',t') \\
& = 
\Big(x_1+x_1', y_1+y_1', \dots , x_n + x_n' , y_n+y_n', t+t'+2 \sum_{j=1}^n (x_j' y_j - x_j y_j')\Big)
\end{align*}
with Lie algebra $\mathfrak{g}$ whose basis of left invariant vector fields is
$$
X_j(p) = \frac{\partial}{\partial x_j} + 2y_j \frac{\partial}{\partial t}, 
\quad 
Y_j(p)= \frac{\partial}{\partial y_j} - 2x_j \frac{\partial}{\partial t}, 
\quad 
T=\frac{\partial}{\partial t},
\quad
j=1,2,\ldots,n
$$
at any $p=(x_1, y_1, \dots, x_n , y_n,t) \in \mathbb{H}^n$.
We call $H \mathbb{H}^n = \mathrm{span} \{ X_1,Y_1,\dots, X_n, Y_n \}$ the {\em horizontal distribution} on $\mathbb{H}^n$, 
and denote by $H_p \mathbb{H}^n$ the horizontal space at $p$.
It is easy to see that the horizontal distribution is the kernel of the {\em standard contact form}
\begin{equation}
\label{contact}
\alpha = dt+2 \sum_{j=1}^n(x_j dy_j - y_j dx_j).
\end{equation}
That is, $H_p \mathbb{H}^n = \text{ker} \, \alpha (p)$. 
We say that an absolutely continuous curve
$\gamma:[a,b] \to \mathbb{R}^{2n+1}$ is \emph{horizontal} if $\gamma'(t) \in H_{\gamma(t)} \mathbb{H}^n$
for almost every $t \in [a,b]$.

Equip the horizontal distribution $H \mathbb{H}^n$ with the left invariant metric 
which makes all of the vectors $X_j$ and $Y_j$ orthonormal at every point in $\mathbb{H}^n$.
Under this metric, if we write
$$
\gamma'(t) = \sum_{j=1}^n \alpha_j(t) X_j(\gamma(t)) + \beta_j(t) Y_j(\gamma(t)) \quad \text{ for a.e. } t \in [a,b],
$$
for any horizontal curve $\gamma: [a,b] \to \mathbb{H}^n$,
then the length of $\gamma$ is
$$
\ell_H(\gamma) := \int_a^b \Vert \gamma'(t)\Vert _H \, dt = \int_a^b \sqrt{\sum_{j=1}^n \alpha_j^2(t) + \beta_j^2(t)}\, dt.
$$
If we write $\pi:\mathbb{R}^{2n+1} \to \mathbb{R}^{2n}$ for the projection onto the first $2n$ coordinates,
notice that $\ell_H(\gamma)$ is equal to the Euclidean length $\ell_E(\pi \circ \gamma)$.
Therefore,
\begin{equation}
\label{lengthbound}
\ell_H(\gamma) \leq \ell_E(\gamma).
\end{equation}

We equip $\mathbb{H}^n$ with the Carnot-Carath\'{e}odory metric $d$ 
defined so that $d(p,q)$ equals the infimum of lengths $\ell_H(\gamma)$ over all horizontal curves $\gamma$ connecting $p$ and $q$.
Any two points in $\mathbb{H}^n$ may be connected by a horizontal curve of finite length, 
so $d$ is indeed a metric.
Topologically, $(\mathbb{H}^n,d)$ is homeomorphic to $\mathbb{R}^{2n+1}$.
Moreover, for any compact $K \subset \mathbb{H}^n$, there is a constant $C \geq 1$ so that
\begin{equation}
\label{equivalent}
C^{-1}|p-q| \leq d(p,q) \leq C|p-q|^{1/2}
\end{equation}
for every $p,q \in K$.
In particular, for any $E \subset \mathbb{R}^m$, 
every locally Lipschitz mapping $F:E \to \mathbb{H}^n$ is also locally Lipschitz as a mapping into $\mathbb{R}^{2n+1}$.
Moreover, one may show that $F$ is bounded as a mapping into $\mathbb{R}^{2n+1}$
if and only if
it is bounded as a mapping into $\mathbb{H}^n$.

It will occasionally be helpful for us to consider a different but bi-Lipschitz equivalent metric $d_K$ on $\mathbb{H}^n$
called the \emph{Kor\'{a}nyi metric} defined for any $p,q \in \mathbb{H}^n$ as
$$
d_K(p,q) = \Vert q^{-1} * p\Vert _K \quad \text{ where } \quad \Vert (x,y,t)\Vert _K = \left(|(x,y)|^4 + t^2 \right)^{1/4}.
$$
If we write $p = (x,y,t)$ and $q = (x',y',t')$, then
\begin{align*}
d_K(p,q) 
&= \left( \left[ \sum_{j=1}^n (x_j-x_j')^2 + (y_j-y_j')^2 \right]^2 + \left[ t-t' + 2 \sum_{j=1}^n (x_j'y_j - x_j y_j') \right]^2 \right)^{1/4} \\
&\approx \left[ \sum_{j=1}^n (x_j-x_j')^2 + (y_j-y_j')^2 \right]^{1/2} + \left| t-t' + 2 \sum_{j=1}^n (x_j'y_j - x_j y_j') \right|^{1/2}
\end{align*}
where $f \approx g$ means $C^{-1}f \leq g \leq Cf$ for some constant $C \geq 1$.
In particular, the above relationship combined with the bi-Lipschitz equivalence of $d$ and $d_K$ gives
\begin{equation}
\label{koranyi}
\left| t-t' + 2 \sum_{j=1}^n (x_j'y_j - x_j y_j') \right|^{1/2} \leq C d(p,q)
\end{equation}
for some constant $C \geq 1$.
For more details about the Heisenberg group and proofs of the above claims, see \cite{heisenberg_book}.

Finally, we will use the following result in the proof of Proposition~\ref{sharp}.
This is a result from \cite{rekt}, and another construction is given in \cite[Theorem 3.2]{haj_lack}.
\begin{theorem}
\label{embedding}
For any $n \geq 1$,
there is a smooth embedding of the sphere $S^n$ into $\mathbb{R}^{2n+1}$
which is horizontal and bi-Lipschitz as a mapping into $\mathbb{H}^n$
and has no Lipschitz extension $F:B^{n+1} \to \mathbb{H}^n$.
\end{theorem}

\section{Sobolev mappings into metric spaces}
\label{sec_sob}

For a domain $\Omega \subset \mathbb{R}^m$,
the Sobolev space $W^{1,p}(\Omega)$, $1 \leq p < \infty$, consists of those real valued functions in $L^p(\Omega)$ 
whose distributional partial derivatives are also functions in $L^p(\Omega)$.
The Sobolev space $W^{1,p}(\Omega,\mathbb{R}^{k})$ consists of 
mappings from $\Omega$ into $\mathbb{R}^{k}$ whose components are members of $W^{1,p}(\Omega)$.
For any $f \in W^{1,p}(\Omega)$, we write $\nabla f$ 
to denote the vector consisting of the $m$ weak partial derivatives of $f$.

The following classical characterization of Sobolev functions 
will be used several times throughout the paper.
Suppose $\Omega$ is a domain in $\mathbb{R}^m$.
Call $ACL(\Omega)$ the space of all measurable real valued functions $u$ on $\Omega$
so that, for $(m-1)$-almost every line $\bar{\ell}$ parallel to a coordinate axis,
the restriction of $u$ to $\ell = \bar{\ell} \cap \Omega$ is locally absolutely continuous.
In particular, the partial derivatives of $u$ exist almost everywhere in $\Omega$ in the classical sense.
Say $u \in ACL^p(\Omega)$ if $u \in ACL(\Omega)$ and $u,|\nabla u| \in L^p(\Omega)$.
Say $u \in ACL^p(\Omega,\mathbb{R}^n)$ if each of the component functions of $u$ is in $ACL^p(\Omega)$.

\begin{lemma}
\label{aclp}
Suppose $1 \leq p < \infty$.
Then $W^{1,p}(\Omega) = ACL^p(\Omega)$.
\end{lemma}

For a proof, see \cite[Theorem 2.1.4]{ziemer}.
More precisely, if $u \in W^{1,p}(\Omega)$, 
then there is some representative $\tilde{u}$ of $u$ for which $\tilde{u} \in ACL^p(\Omega)$.
Conversely, if $u \in ACL^p(\Omega)$,
then $u \in W^{1,p}(\Omega)$, and the weak partial derivatives of $u$ 
equal the classical partial derivatives almost everywhere.

The following definition of Sobolev mappings into the Heisenberg group has been discussed in \cite{haj_weak, haj_lack, peano, mag_mal_mon}.
The class $W^{1,p}(\Omega,\mathbb{H}^n)$ is defined differently in these references, but the definitions are proven to be equivalent in \cite[Proposition 6.8]{haj_lack}.

\begin{definition}
\label{definition}
Suppose $\Omega$ is a bounded domain in $\mathbb{R}^m$.
A mapping $F:\Omega \to \mathbb{H}^n$ 
is in the class $W^{1,p}(\Omega,\mathbb{H}^n)$ if the following two conditions hold:
\begin{enumerate}
\item $F \in W^{1,p}(\Omega,\mathbb{R}^{2n+1})$, and \label{clarity}
\item $F=(f_1,g_1,\dots,f_n,g_n,h)$ satisfies the \emph{weak contact equation}
\begin{equation}
\label{weak_contact}
\nabla h (x) = 2 \sum_{j=1}^n \left( g_j(x) \nabla f_j (x) - f_j(x) \nabla g_j (x) \right) \quad \text{a.e. } x \in \Omega.
\end{equation}
\end{enumerate}
Say that $F \in W^{1,p}_{loc}(\mathbb{R}^m,\mathbb{H}^n)$ 
if $F \in W^{1,p}_{loc}(\mathbb{R}^m,\mathbb{R}^{2n+1})$ and the weak contact equation holds for a.e. $x \in \mathbb{R}^m$.
\end{definition}

For clarification, 
item (\ref{clarity}) here means that the mapping $F$ belongs to an equivalence class of mappings
in the Banach space $W^{1,p}(\Omega,\mathbb{R}^{2n+1})$.
In particular, $W^{1,p}(\Omega,\mathbb{H}^n)$ is a collection of mappings 
rather than a collection of equivalence classes.

Notice that the weak contact condition \eqref{weak_contact} may also be written as follows:
$$
\text{im }DF(x) \subset H_{F(x)}\mathbb{H}^n \quad \text{ for a.e. } x \in \Omega
$$
where $DF$ is the weak differential of $F$.
Consider the projection mapping $\pi$ from $\mathbb{R}^{2n+1}$ onto its first $2n$ coordinates.
It follows from the definition of the metric on the horizontal space that 
$d \pi(p):H_{p}\mathbb{H}^n \to T_{\pi(p)} \mathbb{R}^{2n}$
is an isometry for any $p \in \mathbb{H}^n$.
Hence, for almost every $x \in \Omega$,
the norm of the linear map $DF(x):T_x \mathbb{R}^m \to H_{F(x)}\mathbb{H}^n$ 
is equal to the norm of $D(\pi \circ F)(x):T_x \mathbb{R}^m \to T_{\pi(F(x))} \mathbb{R}^{2n}$.
This is why the quantitative estimates at the end of the statement of Theorem~\ref{main} 
only apply to the partial derivatives of the first $2n$ components of $F$.

As we will now see, this definition gives a sufficient condition
for a mapping to be in the class $AR^{1,p}(\Omega,\mathbb{H}^n)$.
Recall the following definition of the Ambrosio-Reshetnyak-Sobolev class from the introduction.
For $1 \leq p < \infty$ and a bounded domain $\Omega$ in $\mathbb{R}^m$,
a mapping $F: \Omega \to Y$ belongs to the class $AR^{1,p}(\Omega,Y)$
if there is a non-negative function $g \in L^p(\Omega)$ satisfying the following:
for any $K$-Lipschitz $\phi:Y \to \mathbb{R}$, 
we have $\phi \circ F \in W^{1,p}(\Omega)$ 
and $|\partial (\phi \circ F) / \partial x_k(x)| \leq K g(x)$ 
for $k=1,\dots,m$ and almost every $x \in \Omega$.
As above, $AR^{1,p}(\Omega,Y)$ is a collection of mappings rather than a collection of equivalence classes.

\begin{proposition}
\label{heisenberg_sobolev}
Suppose $\Omega$ is a bounded domain in $\mathbb{R}^m$ and $1 \leq p < \infty$.
Then $W^{1,p}(\Omega,\mathbb{H}^n) \subset AR^{1,p}(\Omega,\mathbb{H}^n)$.
Furthermore, if $F \in AR^{1,p}(\Omega,\mathbb{H}^n)$ is bounded, 
then $F \in W^{1,p}(\Omega,\mathbb{H}^n)$.
\end{proposition}

A result similar to the first inclusion was proven in \cite[Proposition 6.1]{haj_weak}
by embedding $\mathbb{H}^n$ into $\ell^{\infty}$ via the Kuratowski embedding.
The reverse inclusion for bounded maps is proven in \cite[Proposition 6.8]{haj_lack}
by applying the same embedding and invoking an ACL-type result for Sobolev mappings into Banach spaces.
Different, mostly self-contained proofs relying more directly on the geometry of the Heisenberg group are given below.

\begin{proof}
Suppose $F:\Omega \to \mathbb{H}^n$ is such that
$F \in W^{1,p}(\Omega,\mathbb{R}^{2n+1})$ and satisfies \eqref{weak_contact} almost everywhere in $\Omega$.
Without loss of generality, 
we may assume that $F \in ACL^p(\Omega,\mathbb{R}^{2n+1})$.
Indeed, if $\tilde{F} \in AR^{1,p}(\Omega,\mathbb{H}^n)$ and $F = \tilde{F}$ almost everywhere in $\Omega$, 
then $F \in AR^{1,p}(\Omega,\mathbb{H}^n)$.

Fix a $K$-Lipschitz function $\phi:\mathbb{H}^n \to \mathbb{R}$.
First, notice for any $x \in \Omega$  
$$
|\phi(F(x))| 
\leq K \, d(F(x),0) + |\phi(0)|
\leq C \, K \, \Vert F(x)\Vert _K + |\phi(0)|
$$
for some $C \geq 1$ from the bi-Lipschitz equivalence of $d$ and $d_K$.
There is a constant $M \geq 1$ depending only on $n$ 
so that $\Vert p\Vert _K \leq M \, \max \{1, |p| \}$ for any $p \in \mathbb{H}^n$.
Hence, since $\Omega$ is bounded and $F \in L^p(\Omega,\mathbb{R}^{2n+1})$,
we have $\phi \circ F \in L^p(\Omega)$.

We must now show that $\phi \circ F \in W^{1,p}(\Omega)$ and find a function $g \in L^{p}(\Omega)$
which dominates the partial derivatives of $\phi \circ F$ and is independent of the choice of $\phi$.
Fix $k \in \{1,\dots,m\}$.
Choose a line $\bar{\ell}$ parallel to the $k^{th}$ coordinate axis so that 
$F$ is absolutely continuous along compact intervals in $\ell := \bar{\ell} \cap \Omega$
and so that $\partial F / \partial x_k \in L^p(\ell,\mathbb{R}^{2n+1})$.
Suppose also that $F$ satisfies \eqref{weak_contact} almost everywhere along $\ell$.
(Note that $(m-1)$-almost every $\bar{\ell}$ parallel to the $k^{th}$ coordinate axis satisfies these conditions via Fubini's theorem and Lemma \ref{aclp}.)
Choose a compact interval $[a,b] \subset \ell$.
(Here, we abuse notation and identify $\ell$ with a subset of $\mathbb{R}$.)
It follows from \eqref{weak_contact} that $\gamma:= F|_{[a,b]}:[a,b] \to \mathbb{H}^n$ is a horizontal curve.
The definition of the metric in $\mathbb{H}^n$ and \eqref{lengthbound} give
$$
|\phi(F(x)) - \phi(F(y))| 
\leq K \, d(F(x),F(y)) 
\leq K \, \ell_H (\gamma|_{[x,y]}) 
\leq K \, \ell_E (\gamma|_{[x,y]})
$$
for any $[x,y] \subset [a,b]$.
Consider the Euclidean length function $s_{\gamma}:[a,b] \to [0,\ell_E (\gamma)]$ defined as
$
s_{\gamma} (x) = \ell_E (\gamma|_{[a,x]})
$.
We can write $\ell_E(\gamma|_{[x,y]}) = | s_{\gamma}(x) - s_{\gamma}(y) |$
and conclude that 
$$
|\phi(F(x)) - \phi(F(y))|
\leq K \, | s_{\gamma}(x) - s_{\gamma}(y) |
$$
for any $x,y \in [a,b]$.
Since $\gamma$ is absolutely continuous on $[a,b]$
as a Euclidean curve, 
$s_{\gamma}$ is absolutely continuous as well (see for example \cite[Proposition 5.1.5]{sobolev_book}).
Thus $\phi \circ F$ is absolutely continuous on $[a,b]$.

We will now prove the bound on the derivative of $\phi \circ F$ along $\ell$.
Fix a point $x \in \ell$ where $\partial F / \partial x_k$ and $\partial (\phi \circ F) / \partial x_k$ exist
and which is a $p$-Lebesgue point of each component of $\partial F / \partial x_k$.
(Note: almost every point in $\ell$ satisfies these conditions since the partial derivative of $F$ is $p$-integrable along $\ell$.)
For any $t$ small enough so that the interval $(x,x+ te_k) \subset \Omega$, we have
\begin{align*}
&\left| \frac{\phi(F(x+te_k)) - \phi(F(x))}{t} \right|
\leq C \, K \, \frac{d_K(F(x+te_k),F(x))}{|t|} \\
& \hspace{.2in} = C \, K 
\Bigg( \left| \sum_{j=1}^n \left( \frac{f_j(x+te_k)-f_j(x)}{t} \right)^2 + \left(\frac{g_j(x+te_k)-g_j(x)}{t} \right)^2 \right|^2 \\
& \hspace{1in} + \left| \frac{h(x+te_k)-h(x) + 2 \sum_{j=1}^n (f_j(x)g_j(x+te_k) - f_j(x+te_k) g_j(x))}{t^2} \right|^2 \Bigg)^{1/4} \\
\end{align*}
for a constant $C>0$ depending only on the bi-Lipschitz equivalence of $d$ and $d_K$.
This final fraction above converges to $0$ as $t \to 0$.
Indeed, the proof of this fact is nearly identical to the proof of Proposition 1.4 in \cite{ME}
since $x$ is a $p$-Lebesgue point of the partial derivatives.
Therefore, 
\begin{equation}
\label{partialbound}
\left| \frac{\partial (\phi \circ F)}{\partial x_k} (x)\right|
\leq C \, K \, \sqrt{ \sum_{j=1}^n \left( \frac{\partial f_j}{\partial x_k}(x) \right)^2 + \left(\frac{\partial g_j}{\partial x_k}(x) \right)^2 }
\leq C \, K \, \left| \frac{\partial F}{\partial x_k}(x) \right|.
\end{equation}
Define $g:\Omega \to \mathbb{R}$ as $g(x)=C \sum_{k=1}^m \left| \frac{\partial F}{\partial x_k}(x) \right|$.
Thus, for any $K$-Lipschitz $\phi:\mathbb{H}^n \to \mathbb{R}$,
we have $|\partial (\phi \circ F) / \partial x_k(x)| \leq K g(x)$ 
for almost every $x \in \Omega$ and $k=1,\dots,m$.
Since $g \in L^p(\Omega)$,
it follows that $F \in AR^{1,p}(\Omega,\mathbb{H}^n)$.

We will now prove the reverse inclusion for bounded Sobolev mappings.
Suppose $F \in AR^{1,p}(\Omega,\mathbb{H}^n)$ is bounded
and say $g \in L^p(\Omega)$ is as in the definition of the Ambrosio-Reshetnyak-Sobolev class.
By \eqref{equivalent}, the identity map $\text{id}:\mathbb{H}^n \to \mathbb{R}^{2n+1}$ 
is Lipschitz on some compact set containing $F(\Omega)$.
Thus $F = \text{id} \circ F \in W^{1,p}(\Omega, \mathbb{R}^{2n+1})$.
It remains to show that the weak contact equation \eqref{weak_contact} holds almost everywhere.
Choose a dense subset $\{p_i\}_{i=1}^{\infty}$ of $\mathbb{H}^n$.
(This is possible since $\mathbb{H}^n$ and $\mathbb{R}^{2n+1}$ are topologically equivalent.)
Define the 1-Lipschitz maps
$\phi_{i}: \mathbb{H}^n \to \mathbb{R}$
as $\phi_{i}(x) = d(x,p_i)$.
Therefore, in $\Omega$ along $(m-1)$-almost every line parallel to a coordinate axis, 
$\phi_i \circ F$ is absolutely continuous (after possibly redefining $F$ on a set of measure zero),
$g$ is $p$-integrable,
and $|\partial (\phi_i \circ F) / \partial x_k| \leq g$ almost everywhere 
for all $i \in \mathbb{N}$.
For $k \in \{ 1, \dots, m\}$,
fix such a line $\bar{\ell}$ parallel to the $k^{th}$ axis 
and write $\ell = \bar{\ell} \cap \Omega$.

By Fubini's theorem, it suffices to prove that \eqref{weak_contact} holds almost everywhere along $\ell$.
Choose an interval $[x,x+te_k] \subset \ell$.
Fix $s_1,s_2 \in [0,t]$. 
Let $\varepsilon > 0$ and choose $p_i \in \mathbb{H}^n$
so that $2d(F(x+s_1e_k),p_i) < \varepsilon$.
Then we have
\begin{align*}
d(F(x+ s_2e_k),F(x+s_1e_k)) - \varepsilon
&\leq d(F(x+ s_2e_k),F(x+s_1e_k)) - 2 d(F(x+ s_1e_k), p_i) \\
&\leq d(F(x+ s_2e_k),p_i) - d(F(x+ s_1e_k), p_i) \\
&=\phi_i(F(x+ s_2e_k)) - \phi_i(F(x+s_1e_k)) \\
&= \int_{s_1}^{s_2} \frac{d}{d \tau} \, (\phi_i \circ F)(x+\tau e_k) \, d \tau \\
&\leq \int_{s_1}^{s_2} g(x+\tau e_k) \, d\tau.
\end{align*}
Since $\varepsilon >0$ was chosen arbitrarily, it follows that
$$
d(F(x+ s_2e_k),F(x+s_1e_k)) \leq \int_{s_1}^{s_2} g(x+\tau e_k) \, d\tau
$$
for any $s_1,s_2 \in [0,t]$.
By the integrability of $g$ along $\ell$,
the mapping $F$ is absolutely continuous with respect to the metric $d$ along compact intervals in $\ell$.
Hence \eqref{weak_contact} holds almost everywhere along $\ell$
as a result of Proposition 4.1 in \cite{pansu}.
This completes the proof of the proposition.
\end{proof}

Notice in \eqref{partialbound} that only the first $2n$ components of $F$
appear in the bound of the partial derivatives of $\phi \circ F$.
Compare this to the bound in Theorem~\ref{main}
and to the discussion following Definition~\ref{definition}.

We will conclude the section with the proof of Proposition~\ref{sharp}.

\begin{proof}[Proof of Proposition \ref{sharp}]
Define $f:S^n \to \mathbb{H}^{n}$ to be the embedding from Theorem~\ref{embedding}.
Suppose we have a mapping $F:B^{n+1} \to \mathbb{H}^n$ 
satisfying $F \in W^{1,n+1}(B^{n+1},\mathbb{H}^n)$ and $F|_{S^n} = f$.
By the definition of $W^{1,n+1}(B^{n+1}, \mathbb{H}^n)$ and Theorem 1.4 in \cite{haj_weak},
$\text{rank} \, DF(x) \leq n$ for almost every $x \in B^{n+1}$.
Since $f^{-1}:f(S^n) \to S^n$ is $C^1$, 
we may find a $C^1$ extension $\Psi:\mathbb{R}^{2n+1} \to \mathbb{R}^{n+1}$ of $f^{-1}$
so that $|D \Psi| \leq M$ for some $M>0$.
Now, choose a sequence $\{F_k\}$ of mappings $F_k:B^{n+1} \to \mathbb{R}^{2n+1}$ which are $C^1$ up to the boundary
and which satisfy the following:
\begin{itemize}
\item $\Vert F_k - F\Vert _{W^{1,n+1}} \to 0$ as $k \to \infty$,
\item $\mathcal{H}^{n+1}(\{ F_k \neq F \}) \to 0$ as $k \to \infty$,
\item and $F_k = F = f$ on $S^n$ for any $k \in \mathbb{N}$
\end{itemize}
(see, for example, Theorem 5 and the proof of Theorem 2 in \cite{haj_geo}.)
Fix $k \in \mathbb{N}$.
Since $\Psi \circ F_k$ is continuous on $B^{n+1}$ and equals the identity map on $S^n$,
Brouwer's theorem implies $B^{n+1} \subset (\Psi \circ F_k)(B^{n+1})$.
Additionally, $|J(\Psi \circ F_k)| \leq M|J F_k|$.
Here, the Jacobian $|J F_k|$ is understood in the following sense:
$$
|J F_k(x)| = \sqrt{\det \, ((DF_k)^T DF_k) (x)} \quad \text{for all } x \in B^{n+1}.
$$
Thus
$$
M \int_{B^{n+1}} |JF_k| \geq \int_{B^{n+1}} |J(\Psi \circ F_k)| \geq \mathcal{H}^{n+1}((\Psi \circ F_k)(B^{n+1})) \geq \mathcal{H}^{n+1}(B^{n+1}).
$$
Since $\text{rank} \, DF(x) \leq n$ for almost every $x \in B^{n+1}$,
it follows that $|JF_k| = 0$ almost everywhere on $\{F_k = F\}$.
Therefore
$$
0 
< \frac{\mathcal{H}^{n+1}(B^{n+1})}{M} 
\leq \int_{B^{n+1}} |JF_k|
= \int_{\{ F_k \neq F \}} |JF_k|.
$$
However, $\mathcal{H}^{n+1}(\{F_k \neq F\}) \to 0$, and
$|JF_k|$ converges to $|JF|$ in $L^1$
due to the convergence of $F_k$ to $F$ in $W^{1,n+1}$
since the Jacobian consists of sums of $(n+1)$-fold products of derivatives.
Thus this last integral vanishes as $k \to \infty$.
This leads to a contradiction and completes the proof.
\end{proof}

\section{Whitney triangulation and Lipschitz extensions}
\label{AN}

Suppose $Z \subset \mathbb{R}^m$ is closed.
As in the proof of many extension theorems,
we will decompose the complement of $Z$ into Whitney cubes.
We will then go one step further and construct the Whitney triangulation of the complement of $Z$ as in \cite{vaisala}.
We must first introduce some notation.
For any $k \in \{0,1,\dots,m\}$, a \emph{(non-degenerate)} $k$\emph{-simplex} in $\mathbb{R}^m$
is the convex hull of $k+1$ \emph{vertices} $\{e_0,e_1,\dots,e_k\} \subset \mathbb{R}^m$ where the vectors
$e_1-e_0,\dots,e_k-e_0$ are linearly independent.
An $\ell$-face $\omega$ of a $k$-simplex $\sigma$ 
is the convex hull of any subset $\{ e_{i_0},\dots,e_{i_{\ell}} \}$ of vertices of $\sigma$.
Denote by $\partial \omega$ the union of all $(\ell - 1)$-faces of $\omega$.
Note that, since we define simplices to be nondegenerate, 
the barycenter of a simplex does not lie in any of its faces.
A \emph{simplicial complex} $\Sigma$ in $\mathbb{R}^m$ is a (possibly infinite) set 
consisting of simplices in $\mathbb{R}^m$ so that 
any face of a simplex in $\Sigma$ is an element of $\Sigma$ 
and the intersection of any two simplices in $\Sigma$ is either empty or is itself an element of $\Sigma$.
The dimension of $\Sigma$ is the largest $k$ so that $\Sigma$ contains a $k$-simplex.
(Notice that the dimension of a simplicial complex in $\mathbb{R}^m$ is at most $m$.)
For any $k \in \{0,1,\dots,m\}$, the $k$\emph{-skeleton} of $\Sigma$ (denoted $\Sigma^{(k)}$) 
is the subset of $\mathbb{R}^m$ consisting of the union of all $k$-simplices in $\Sigma$.
Similarly, the $\ell$\emph{-skeleton} $\Sigma_{\sigma}^{(\ell)}$ of a $k$-simplex $\sigma$, $0\leq \ell \leq k$, is the union of all $\ell$-faces of $\sigma$.
Finally, we will write $B(k,\ell) := \binom{k+1}{\ell+1}$.
This is the number of $\ell$-faces of a $k$-simplex.

Suppose $\Sigma$ is a simplicial complex in $\mathbb{R}^m$.
For each $\ell \in \{ 1,\dots,m \}$ and any $\ell$-simplex $\omega \in \Sigma$ with barycenter $c$, say 
$\beta(\omega)$ is the minimum over all distances $d(c,P)$ where $P$ is an $(\ell-1)$-plane containing an $(\ell-1)$-face of $\omega$.
In particular, $\beta(\omega)>0$.
Similarly, say $B(\omega)$ is the maximum over all such distances.
For any $m$-simplex $\sigma$, write 
$$
\beta_{\sigma} = \min \left\{ \beta(\omega) \, : \, \omega \text{ is an } \ell \text{-face of } \sigma \text{ for some } \ell \in \{ 1 ,\dots, m\} \right\}
$$
and
$$
B_{\sigma} = \max \left\{ B(\omega) \, : \, \omega \text{ is an } \ell \text{-face of } \sigma \text{ for some } \ell \in \{ 1 ,\dots, m\} \right\}.
$$
That is, $\beta_{\sigma}$ is a lower bound on the ``flatness'' of $\sigma$, 
and $B_{\sigma}$ is an upper bound.
We are now ready to define the Whitney triangulation of $\mathbb{R}^m \setminus Z$.
This lemma is a minor modification of the results in \cite[Section 5.1]{vaisala}.
\begin{lemma}[Whitney Triangulation]
\label{triangulate}
Suppose $Z \subset \mathbb{R}^m$ is closed.
Then there is an $m$-dimensional simplicial complex $\Sigma$ in $\mathbb{R}^m$ so that 
$\Sigma^{(m)} = \mathbb{R}^m \setminus Z$
and the following hold for some constants $D_1,D_2>0$ (which depend only on $m$) 
and any $m$-simplex $\sigma \in \Sigma$:
\begin{equation}
\label{size}
\diam(\sigma) 
\leq d(\sigma,Z) 
\leq 12 \sqrt{m} \, \diam(\sigma),
\end{equation}
\begin{equation}
\label{flat}
D_1 < \frac{\diam(\sigma)}{B_{\sigma}} \leq \frac{\diam(\sigma)}{\beta_{\sigma}} < D_2.
\end{equation}
\end{lemma}

Intuitively, the second condition here implies that the simplices in $\Sigma$ are uniformly far from being degenerate.

\begin{proof}
As in \cite{grafakos}, 
there is a decomposition of the open set $\mathbb{R}^m \setminus Z$ into a family of 
closed dyadic cubes $\{ Q_i \}$ with pairwise disjoint interiors
so that 
\begin{enumerate}[({A}1)]
\item $\bigcup_{i=1}^{\infty} Q_i = \mathbb{R}^m \setminus Z$,
\item $\diam(Q_i) \leq d(Q_i,Z) \leq 4 \diam(Q_i)$ for every $i \in \mathbb{N}$,
\item for any $i \in \mathbb{N}$, at most $12^m$ cubes $Q_j$ intersect $Q_i$ nontrivially.
\end{enumerate}
From this cubic decomposition, 
we will construct the Whitney triangulation inductively as in \cite{vaisala}.
The collection of the vertices of the cubes is trivially a 0-dimensional simplicial complex $\Sigma_0$.
We define $\Sigma_1$ by dividing each edge of a Whitney cube into two 1-dimensional simplices (segments) 
at its midpoint.
Fix $k \in \{2,\dots,m\}$, and suppose a simplicial complex $\Sigma_{k-1}$ has been constructed 
on the union of the $(k-1)$-cubes by dividing them into simplices.
Choose some $k$-cube $Q$ in the Whitney decomposition.
The union of the faces of $Q$ is the $k$-skeleton of a subcomplex of $\Sigma_{k-1}$.
(Recall that the $k$-skeleton is a subset of $\mathbb{R}^m$ rather than a subset of the simplicial complex.)
For each $(k-1)$-simplex in this subcomplex, 
create a $k$-simplex by appending the center of $Q$ to the set of its vertices.
This provides a simplicial subdivision of $Q$ 
and thus a simplicial complex $\Sigma_k$ on the union of the $k$-cubes.
Continuing in this way creates $\Sigma = \Sigma_m$.

Condition \eqref{size} follows immediately from (A2) since, for any $m$-cube $Q$,
the diameter of an $m$-simplex in $Q$ is at least half of the side length of $Q$.
We will say that two simplices in $\Sigma$ are equivalent if one can be obtained from the other 
via a rotation, translation, and homothetic dilation.
There are only finitely many equivalence classes of simplices in $\Sigma$ as a result of (A3).
Since $\diam(\sigma)/B_{\sigma}$ and $\diam(\sigma)/\beta_{\sigma}$ are invariant under 
rotations, translations, and homothetic dilations, we have \eqref{flat}.
\end{proof}

The following Lipschitz extension result will be essential to the construction in the proof of Theorem~\ref{main2}.
Though the proof of this extension lemma is elementary and similar to classical results
(see for example \cite{almgren,extend_banach}), 
it is included here for completeness.
Recall that a metric space $Y$ is 
Lipschitz $(n-1)$-connected if there is a constant $\gamma \geq 1$ so that 
any $L$-Lipschitz map $f:S^k \to Y$ ($L>0$) has a 
$\gamma L$-Lipschitz extension $F:B^{k+1} \to Y$ for $k =0,1,\dots,n-1$. 

\begin{lemma}
\label{wengeryoung}
Fix positive integers $m > n$.
Suppose $Y$ is Lipschitz $(n-1)$-connected with constant $\gamma$, 
and $Z \subset \mathbb{R}^m$ is closed. 
Say $\Sigma$ is the Whitney triangulation of $\mathbb{R}^m \setminus Z$
constructed in Lemma~\ref{triangulate}.
Then there is a constant $\tilde{C} \geq 1$ depending only on $m$, $n$, and $\gamma$
such that every $L$-Lipschitz map $f:Z \to Y$
has an extension $\tilde{f}:Z \cup \Sigma^{(n)} \to Y$
satisfying the following:
\begin{enumerate}
\item $\tilde{f}$ is $L\tilde{C}$-Lipschitz on any $n$-simplex in $\Sigma$, and
\item for any $a \in \Sigma^{(0)}$, $\tilde{f}(a) = f(z_a)$ for some
$z_a \in Z$ with $|z_a - a| = d(a,Z)$.
\end{enumerate}
\end{lemma}

\begin{proof}
Fix an $L$-Lipschitz map $f:Z \to Y$.
For each $a \in \Sigma^{(0)}$ (that is, each vertex of a simplex in $\Sigma$), 
choose a nearest point $z_a \in Z$ 
i.e. $|z_a - a| = d(a,Z)$.
Define the mapping $f^{(0)}:\Sigma^{(0)} \to Y$ as $f^{(0)}(a) := f(z_a)$.
Write $C_0 := D_2(12 \sqrt{m}+1) + 1$ where $D_2$ is the constant from condition \eqref{flat} in Lemma \ref{triangulate}.
Fix a $1$-simplex $\sigma^1$ in $\Sigma$ (that is, an edge of some $m$-simplex $\sigma$).
Write $\partial \sigma^1 = \{a,b\}$.
Then 
\begin{align*}
d(f^{(0)}(a),f^{(0)}(b)) 
&= d(f(z_a),f(z_b)) 
\leq L |z_a - z_b| 
\leq L(|z_a - a| + |z_b-b| + |a-b|) \\
&=L( d(a,Z) + d(b,Z) + |a-b|) 
\leq L( 2 \, d(\sigma,Z) + 2 \, \diam(\sigma) + |a-b|) \\
&\leq L((24 \sqrt{m} + 2) \, \diam(\sigma) + |a-b|) < L (D_2(12 \sqrt{m} + 1) +1) |a-b|.
\end{align*}
since $\beta_{\sigma} \leq \frac{1}{2}|a-b|$.
That is, $f^{(0)}$ is $LC_0$-Lipschitz continuous on $\partial \sigma^1$.

By the Lipschitz connectivity of $Y$, there is a constant $C_1 >0$ depending only on $C_0$, $n$, 
and $\gamma$ (and hence only on $m$, $n$, and $\gamma$)
and an $L C_1$-Lipschitz extension $f^{(1)}:\sigma^1 \to Y$ of $f^{(0)}$.
Since the intersection of any two $1$-simplices in $\Sigma$ is a vertex or empty, 
we can define a map $f^{(1)}:\Sigma^{(1)} \to Y$ 
which is $LC_1$-Lipschitz on any 1-simplex in $\Sigma$.

Fix $k \in \{2,\dots,n \}$.
Suppose there is a constant $C_{k-1}$ (depending only on $m$, $n$, and $\gamma$)
and a map $f^{(k-1)} : \Sigma^{(k-1)} \to Y$ 
so that $f^{(k-1)}$ is $LC_{k-1}$-Lipschitz on any 
$(k-1)$-simplex in $\Sigma$. 
Choose a $k$-simplex $\sigma^k$ in $\Sigma$.
We will first determine the Lipschitz constant of $f^{(k-1)}$ restricted to $\partial \sigma^{k}$.
Say $x,y \in \partial \sigma^{k}$.
If $x$ and $y$ lie in the same $(k-1)$-face of $\sigma^k$, 
then $d(f^{(k-1)}(x),f^{(k-1)}(y)) \leq L C_{k-1} |x-y|$.
Suppose $x$ and $y$ lie in different $(k-1)$-faces $\sigma_x^{k-1}$ and $\sigma_y^{k-1}$ of $\sigma^k$. 
We have the following simple lemma.

\begin{lemma}
\label{lippoint}
Fix $j \in \{1, \dots, m-1 \}$.
There is a constant $\mu\geq 1$ depending only on $m$ satisfying the following:
suppose $\omega_1$ and $\omega_2$ are $j$-faces of a $(j+1)$-simplex $\omega \in \Sigma$,
and $x \in \omega_1$ and $y \in \omega_2$.
Then there is a point $v \in \omega_1 \cap \omega_2$ so that 
\begin{equation}
\label{point}
|x-v| + |v-y| \leq \mu |x-y|.
\end{equation}
\end{lemma}
\begin{proof}
Choose $v$ to be the orthogonal projection of $x$ or $y$ onto $\omega_1 \cap \omega_2$.
Since there are only finitely many possible angles 
at which the faces of the simplices in the Whitney triangulation can meet,
the law of sines provides a uniform bound for the ratios $|x-v|/|x-y|$ and $|y-v|/|x-y|$. 
That is, we may choose $\mu$ satisfying \eqref{point} 
independent of the choice of faces $\omega_1$ and $\omega_2$ and simplex $\omega$.
\end{proof}

By applying the lemma to the faces $\sigma_x^{k-1}$ and $\sigma_y^{k-1}$ of $\sigma^k$, 
we have
\begin{align*}
d(f^{(k-1)}(x),f^{(k-1)}(y)) 
&\leq d(f^{(k-1)}(x),f^{(k-1)}(v)) + d(f^{(k-1)}(v),f^{(k-1)}(y)) \\
&\leq LC_{k-1}|x-v| + LC_{k-1}|v-y| 
\leq \mu LC_{k-1}|x-y|
\end{align*}
since $f^{(k-1)}$ is $L C_{k-1}$-Lipschitz 
when restricted to each of $\sigma_x^{k-1}$ and $\sigma_y^{k-1}$.
Hence $f^{(k-1)}$ is $\mu LC_{k-1}$-Lipschitz on $\partial \sigma^k$.
Therefore the Lipschitz connectivity of $Y$ gives a
constant $C_k$ depending only on $m$, $n$, $\gamma$, and $C_{k-1}$
and an $LC_k$-Lipschitz extension $f^{(k)}:\sigma^k \to Y$ of $f^{(k-1)}$.
Since the intersection of any two $k$-simplices is a lower dimensional simplex (or empty), 
we may define a mapping $f^{(k)}:\Sigma^{(k)} \to Y$ 
which is $LC_k$-Lipschitz on each $k$-simplex in $\Sigma$.

Continuing this construction inductively gives a constant $C_n$ (depending only on $m$, $n$, and $\gamma$)
and a map $f^{(n)}: \Sigma^{(n)} \to Y$ 
so that $f^{(n)}$ is $LC_n$-Lipschitz on any $n$-simplex in $\Sigma$.
Setting $\tilde{f} := f^{(n)}$ and $\tilde{C} := C_n$ completes the proof.
\end{proof}

\section{Proofs of Theorem \ref{main2} and Theorem \ref{main}}
\label{sec_main}

The proof of Theorem~\ref{main2} is presented here.
We will conclude the section with the proof of Theorem~\ref{main}.
It will follow as a simple consequence of Proposition~\ref{heisenberg_sobolev}
since the extension we construct will be bounded in $\mathbb{H}^n$.

\begin{proof}[Proof of Theorem~\ref{main2}]
Fix $1 \leq p < n+1$ and let $\Omega$ be a bounded domain in $\mathbb{R}^m$.
Suppose $Y$ is a Lipschitz $(n-1)$-connected metric space with constant $\gamma$.
Let $Z \subset \Omega$ be compact and nonempty, and suppose $f:Z \to Y$ is $L$-Lipschitz.

If $m \leq n$, then it can be seen from classical results \cite{almgren,extend_banach}
that there is a constant $C = C(n,\gamma)$ 
and a $CL$-Lipschitz extension $F:\mathbb{R}^m \to Y$ of $f$.
The proof of this fact is similar to the proof of Lemma~\ref{wengeryoung}.
Hence $\phi \circ F$ is $KCL$-Lipschitz for any $K$-Lipschitz function $\phi:Y \to \mathbb{R}$.
Moreover, for $k=1,\dots,m$, $\partial(\phi \circ F) / \partial x_k$ exists and is bounded by $Kg$ almost everywhere in $\Omega$
where $g:\Omega \to \mathbb{R}$ is the constant function $g \equiv CL$.
Thus $F \in AR^{1,p}(\Omega,Y)$,
and $\Vert g\Vert _{L^p(\Omega)} \leq CL|\Omega|^{1/p} \leq CL \left( \diam(\Omega) \right)^{m/p}$
for a constant $C$ depending only on $m$, $n$, and $\gamma$.
We may therefore assume for the remainder of the proof that $m > n$.

Define the Whitney triangulation of $\mathbb{R}^m \setminus Z$ as in Lemma \ref{triangulate}.
We will restrict our attention to the $m$-dimensional simplicial sub-complex $\Sigma$
consisting of those simplices in the Whitney triangulation
which are contained in a Whitney cube $Q$ with $Q \cap \Omega \neq \emptyset$.
We consider this restriction so that 
$
\sup \{ \diam(\sigma) \, | \, \sigma \in \Sigma \} < \infty
$
(since $\Omega$ is bounded).
Note also that $\Omega \setminus Z \subset \Sigma^{(m)}$.

Suppose $\sigma$ is an $m$-simplex in $\Sigma$.
We begin by constructing a sort of radial projection of $\sigma$ onto its $n$-skeleton.
This is the so called zero degree homogenization mentioned in the introduction.
Denote by $c$ the barycenter of $\sigma$.
For each $j \in \{ 1, \dots, m \}$, say $\{ \sigma_i^j \}_{i=1}^{B(m,j)}$ is the collection of $j$-faces of $\sigma$, 
and say $c_i^j$ is the barycenter of $\sigma_i^j$.
(Notice $\sigma_1^m = \sigma$ and $c_1^m = c$.)
Fix $j \in \{ n+1, \dots, m \}$.
For each $i \in \{ 1, \dots, B(m,j) \}$, define $P_i^j:\ \sigma_i^j \setminus \{ c_i^j \} \to \partial \sigma_i^j$ 
to be the projection of $\sigma_i^j \setminus \{c_i^j\}$ onto $\partial \sigma_i^j$ radially out from $c_i^j$.
That is, for $x \in \sigma_i^j \setminus \{c_i^j\}$
if we write $x = c_i^j + t(z-c_i^j)$ 
with $t \in (0,1]$ and $z \in \partial \sigma_i^j$,
then $P_i^j(x) = z$.
Fix $x \in \sigma_i^j \setminus \{c_i^j\}$.
For all $y \in \sigma_i^j \setminus \{c_i^j\}$ close enough to $x$, we have by similar triangles
\begin{equation}
\label{toplip2}
\frac{|P_i^j(x) - P_i^j(y)|}{|x-y|} 
\leq \nu \frac{\diam(\sigma)}{|x-c_i^j|}.
\end{equation}
The constant $\nu>0$ depends only on the dimension $m$
since there are only finitely many equivalence classes of simplices in $\Sigma$.
In particular, $P_i^j$ is locally Lipschitz on $\sigma_i^j \setminus\{c_i^j\}$.
Extend $P_i^j$ to the remaining $j$-skeleton of $\sigma$ by the identity map
(that is, $P_i^j(x) = x$ for any $x \in \Sigma_{\sigma}^{(j)} \setminus \sigma_i^j$).
Writing $C^j=\{ c_1^j, \dots, c_{B(m,j)}^j \}$, we may 
define $P^j : \Sigma_{\sigma}^{(j)} \setminus C^j \to \Sigma_{\sigma}^{(j-1)}$ as $P^j := P_1^j \circ \cdots \circ P_{B(m,j)}^j$.
By arguing in a similar manner to Lemma \ref{lippoint}, 
each $P^j$ is locally Lipschitz on $\Sigma_{\sigma}^{(j)} \setminus C^j$.

In particular, $P^m$ is locally Lipschitz on $\sigma \setminus {c}$.
Now $P^{m-1} \circ P^m$ is defined and locally Lipschitz on $\sigma$ away from the 1-dimensional set
$
\{c\} \cup (P^{m})^{-1}(C^{m-1})
$.
Similarly, 
$
P^{m-2} \circ P^{m-1} \circ P^m
$ 
is locally Lipschitz away from the 2-dimensional set 
$
\{c\} \cup (P^{m})^{-1}(C^{m-1}) \cup (P^{m-1} \circ P^m)^{-1}(C^{m-2})
$.
Continuing in this way, we see that 
$
P_{\sigma}:=P^{n+1} \circ \cdots \circ P^m : \sigma \setminus C_{\sigma} \to \Sigma_{\sigma}^{(n)}
$
is locally Lipschitz off the closed, $(m-n-1)$-dimensional set of singularities
$$
C_{\sigma} := \{c\} \cup \bigcup_{\ell = 1}^{m-(n+1)} (P^{m- \ell + 1} \circ \cdots \circ P^m)^{-1}(C^{m-\ell}).
$$

We will now build the extension $F$ of $f$.
First, construct the extension $\tilde{f}:Z \cup \Sigma^{(n)} \to Y$ of $f$ given in Lemma~\ref{wengeryoung}.
Recall that $\tilde{f}$ is $\tilde{C}L$-Lipschitz on any $n$-simplex in $\Sigma$.
In particular, $\tilde{f}$ is locally Lipschitz on $\Sigma^{(n)}$.
Enumerate the collection of $m$-simplices $\{ \sigma_i \}_{i=1}^{\infty}$ in $\Sigma$,
and write $\mathscr{C} = \bigcup_i C_{\sigma_i}$.
Define $F:\Sigma^{(m)} \cup Z \to Y$ as
$$
 F(x) =
  \begin{cases} 
      \hfill \tilde{f} (P_{\sigma_i}(x)) \hfill & \text{ if $x \in \sigma_i \setminus C_{\sigma_i}$ for some $i \in \mathbb{N}$} \\
      \hfill f(x) \hfill & \text{ if $x \in Z$} \\
  \end{cases}
$$
and define $F$ to be constant on $\mathscr{C}$.
This map is well defined since the intersection 
$\sigma_i \cap \sigma_j$ is either empty or another simplex in $\Sigma$.
Moreover, $F$ is locally Lipschitz on each $\sigma_i \setminus C_{\sigma_i}$.
We now have the following

\begin{lemma}
\label{sobolev}
Suppose $1 \leq p < n+1$.
Define $g: \Sigma^{(m)} \setminus \mathscr{C} \to [0,\infty]$ as
$$
 g(x) = \limsup_{\substack{y \to x}} \frac{d(F(x), F(y))}{|x-y|}.
$$
Then $\Vert g\Vert _{L^p(\Omega \setminus Z)} \leq CL (\diam(\Omega))^{m/p}$ 
for a constant $C>0$ depending only on $m$, $n$, $p$, and $\gamma$.
In particular, $g \in L^p(\Omega \setminus Z)$.
\end{lemma}

The proof of this lemma is long but elementary.
It is contained, therefore, at the end of this section.
Extend $g$ to all of $\Omega$ so that $g \equiv L(\tilde{C}+4)$ on $Z \cup \mathscr{C}$.
Thus $g \in L^p(\Omega)$ and $\Vert g\Vert _{L^p(\Omega)} \leq CL (\diam(\Omega))^{m/p}$
for a constant $C=C(m,n,p,\gamma)$.

It remains to show that $F$ is in the class $AR^{1,p}(\Omega,Y)$.
Fix a $K$-Lipschitz function $\phi:Y \to \mathbb{R}$.
We will first show that $\phi \circ F \in L^p(\Omega)$.
Let $x \in \Omega \setminus (Z \cup \mathscr{C})$.
Then $x \in \sigma_i$ for some $i \in \mathbb{N}$.
Choose a vertex $a$ of $\sigma_i$ so that $a$ and $P_{\sigma_i}(x)$ lie in the same $n$-face of $\sigma_i$.
Since $F(a) = \tilde{f}(a) = f(z_a)$ as prescribed in Lemma~\ref{wengeryoung}, we have
\begin{align*}
|\phi (F(x))| 
\leq |\phi(F(x)) - \phi(F(a))| + |\phi(f(z_a))| 
\leq KL\tilde{C} \, \diam(\sigma_i) + \Vert \phi \circ f\Vert _{\infty} <M
\end{align*}
for some $M>0$.
Since $Z$ is compact and $\Omega$ is bounded, 
$\phi \circ F \in L^p(\Omega)$.

Now, we will use the ACL characterization of Sobolev mappings to show that $\phi \circ F \in W^{1,p}(\Omega)$.
Fix $k \in \{1, \dots, m\}$.
Notice that $(m-1)$-almost every line parallel to the $k^{th}$ coordinate axis
is disjoint from $\mathscr{C}$ since 
each $C_{\sigma_i}$ is $(m-n-1)$-dimensional.
Also, $g$ and $\phi \circ F$ are $p$-integrable in $\Omega$ along $(m-1)$-almost every such line
since $g$ and $\phi \circ F$ are in the class $L^p(\Omega)$.

Choose a line $\bar{\ell}$ parallel to the $k^{th}$ coordinate axis 
that is disjoint from $\mathscr{C}$ 
and suppose that $g \in L^p(\bar{\ell} \cap \Omega)$, 
and $\phi \circ F \in L^p(\bar{\ell} \cap \Omega)$.
Write $\ell := \bar{\ell} \cap \Omega$.
We will now show that $\phi \circ F$ is locally Lipschitz along $\ell \setminus Z$
and its derivative along $\ell \setminus Z$ is $p$-integrable. 
Choose $x \in \ell \setminus Z$.
We need only consider the case when $x \in \partial \sigma_i$ 
for some $i \in \mathbb{N}$ since $F$ is locally Lipschitz on each $\sigma_i \setminus C_{\sigma_i}$.
In this case, for some $a,b \in \ell$, the segments $[a,x]$ and $[x,b]$ 
each lie entirely in some $m$-simplices $\sigma_{a}$ and $\sigma_{b}$ respectively. 
Since $F$ is locally Lipschitz when restricted to each of these simplices, 
it follows that $F$ is Lipschitz along some segment $I \subset [a,b]$ containing $x$. 
Therefore, $F$ is locally Lipschitz on $\ell \setminus Z$,
and hence $\phi \circ F$ is as well.
Now $\partial (\phi \circ F) / \partial x_k$ exists almost everywhere along $\ell \setminus Z$,
and the definition of $g$ gives
$$
\left| \frac{\partial (\phi \circ F)}{\partial x_k}(x) \right| 
\leq K \left[ \limsup_{h \to 0} \frac{d(F(x+h e_k),F(x))}{|h|}\right]
\leq K \, g(x)
$$
for every $x \in \ell \setminus Z$ at which the partial derivative exists.
In particular, $\partial (\phi \circ F) / \partial x_k \in L^p(\ell \setminus Z)$.

Next, we will see that $\phi \circ F$ is in fact continuous along all of $\ell$.
By the previous paragraph, $F$ is continuous along $\ell$ at any $x \in \ell \setminus Z$.
Suppose now that $x \in \ell \cap Z$.
If $y \in \ell \cap Z$, 
then $d(F(x),F(y)) \leq L |x-y|$.
Suppose instead that $y \in \ell \setminus Z$.
Then $y \in \sigma_i$ for some $i \in \mathbb{N}$.
Choose a vertex $a$ of $\sigma_i$ so that $a$ and $P_{\sigma_i}(y)$ lie in the same $n$-face of $\sigma_i$.
Then
\begin{align*}
d(F(y),F(a)) &= d(\tilde{f}(P_{\sigma_i}(y)),\tilde{f}(a)) \\ 
&\leq L\tilde{C}|P_{\sigma_i}(y) - a| \leq L\tilde{C} \diam(\sigma_i) \leq L\tilde{C} d(\sigma_i,Z) \leq L\tilde{C}|x-y|. 
\end{align*}
Also, since $F(a)= f(z_a)$,
\begin{align*}
d(F(a),F(x)) &= d(f(z_a),f(x)) \leq L(|z_a-a|+|a-y|+|y-x|)  \\
&\leq L(d(a,Z) + \diam(\sigma_i) + |x-y|) \\
&\leq L((d(\sigma_i,Z) + \diam(\sigma_i)) + d(\sigma_i,Z) + |x-y| )
< 4L|x-y|.
\end{align*}
Therefore, 
\begin{equation}
\label{lip}
d(F(x),F(y)) \leq L(\tilde{C}+4)|x-y|
\end{equation}
for any $x \in \ell \cap Z$ and $y \in \ell$.
That is, $F$ is continuous on $\ell$, and so $\phi \circ F$ is as well.

Finally, we will show that $\phi \circ F$ is absolutely continuous on any compact interval in $\ell$ as desired.
Since $(\phi \circ f)|_{\ell \cap Z}$ is Lipschitz, 
we may use the classical McShane extension \cite{mcshane}
to find a Lipschitz extension $\psi:\ell \to \mathbb{R}$ of $(\phi \circ f) \big|_{\ell \cap Z}$.
Set $v := (\phi \circ F) - \psi$ on $\ell$.
Notice that $v'$ exists almost everywhere on $\ell \setminus Z$,
and $v' \in L^p(\ell \setminus Z)$.
Moreover, $v$ is continuous on $\ell$, 
is absolutely continuous on compact intervals in $\ell \setminus Z$,
and vanishes on $\ell \cap Z$.
Therefore, by defining
$$
 w(x) =
  \begin{cases} 
      \hfill v'(x)    \hfill & \text{ if $x \in \ell \setminus Z$ and $v'(x)$ exists} \\
      \hfill 0 \hfill & \text{ if $x \in \ell \cap Z$ or $v'(x)$ does not exist}, \\
  \end{cases}
$$
$v$ is the integral of $w$ over any interval in $\ell$.
Since $w$ is integrable on $\ell$, 
it follows that $v$ is absolutely continuous on compact intervals in $\ell$,
and so $\phi \circ F = v + \psi$ is as well.
Therefore, $\phi \circ F \in ACL^{p}(\Omega)$.

Furthermore,
the definition of $g$ together with \eqref{lip} gives
$|\partial(\phi \circ F) / \partial x_k| \leq Kg$ almost everywhere along $\ell$.
Hence, given any $K$-Lipschitz $\phi:Y \to \mathbb{R}$,
we have $\phi \circ F \in W^{1,p}(\Omega)$ and $|\partial (\phi \circ F) / \partial x_k| \leq K \, g$ 
almost everywhere in $\Omega$ for $k = 1 ,\dots, m$.
We may thus conclude that $F \in AR^{1,p}(\Omega,Y)$.
\end{proof}

We are now ready for the proof of Theorem~\ref{main}.
Recall from the discussion in the introduction that $\mathbb{H}^n$ is Lipschitz $(n-1)$-connected \cite{wenger_young}.
According to Proposition~\ref{heisenberg_sobolev}, 
we need only prove that the extension $F$ constructed in the previous proof is bounded as a mapping into $\mathbb{H}^n$
and then prove the desired quantitative estimates.

\begin{proof}[Proof of Theorem~\ref{main}]
Suppose $Y = \mathbb{H}^n$.
Fix $x \in \Omega$.
Notice that $\Vert F(\cdot)\Vert _K$ is bounded on $Z$ since $F \big{|}_Z=f$ is Lipschitz.
Also, $F$ is constant on $\mathscr{C}$.
It therefore suffices to consider $x \in \Omega \setminus (Z \cup \mathscr{C})$.
Hence $x \in \sigma$ for some $m$-simplex $\sigma \in \Sigma$.
Choose a vertex $a$ of $\sigma$ so that $a$ and $P_{\sigma}(x)$ lie in the same $n$-face of $\sigma$.
Then there is some $M>0$ independent of $x$ so that
$$
\Vert F(x)\Vert _K \leq Cd(F(x),F(a)) + \Vert F(a)\Vert _K \leq CL\tilde{C} \diam(\sigma)+\Vert f(z_a)\Vert _K < M
$$
where $C$ is the constant from the bi-Lipschitz equivalence of $d$ and $d_K$.
Thus $F \in AR^{1,p}(\Omega, \mathbb{H}^n)$ is bounded, 
so, by Proposition~\ref{heisenberg_sobolev},
$F \in W^{1,p}(\Omega,\mathbb{H}^n)$.

We now establish the quantitative estimate.
Recall that $\Vert g\Vert _{L^p(\Omega)} \leq CL (\diam(\Omega))^{m/p}$ 
where $g$ was defined in the proof of Theorem \ref{main2}.
Say $\phi_j:\mathbb{H}^n \to \mathbb{R}$ is the projection onto the $j^{th}$ coordinate.
We have that $\phi_j$ is 1-Lipschitz on $\mathbb{H}^n$ for $j=1,\dots,2n$.
Hence the definition of $AR^{1,p}(\Omega,\mathbb{H}^n)$ gives
$|\partial (\phi_j \circ F) / \partial x_k| \leq g$ almost everywhere on $\Omega$
for $j=1,\dots,2n$,
so $\Vert \partial F_j / \partial x_k\Vert _{L^p(\Omega)} \leq CL (\diam(\Omega))^{m/p}$ for $k = 1,\dots,m$ and $j=1,\dots,2n$.

As a final note, we will show that, in fact, $F \in ACL^p(\Omega,\mathbb{R}^{2n+1})$.
Recall that $\phi_j$ is 1-Lipschitz for $j=1,\dots,2n$.
Moreover, $\phi_{2n+1}$ is $K$-Lipschitz on $F(\Omega)$ 
for a constant $K$ depending on $\Omega$ and $F$ (by \eqref{equivalent}).
Use McShane's theorem \cite{mcshane} to extend $\phi_{2n+1}$ to a $K$-Lipschitz function (also called $\phi_{2n+1}$) defined on all of $\mathbb{H}^n$.
According to the proof of Theorem~\ref{main2},
we therefore have $F_j = \phi_j \circ F \in ACL^p(\Omega)$ for each $j$.
\end{proof}

We conclude the paper with the proof of Lemma~\ref{sobolev}.
As mentioned above, this proof is technical but elementary.

\begin{proof}
Suppose $\sigma$ is an $m$-simplex in $\Sigma$.
For the sake of notation, we will write 
$\Phi^k := \tilde{f} \circ P^{n+1} \circ \cdots \circ P^k$ for $k \in \{ n+1,\dots,m \}$
where each $P^k$ is the radial projection of $\Sigma_{\sigma}^{(k)} \setminus C^k$ 
to $\Sigma_{\sigma}^{(k-1)}$ as defined earlier.
As before, for $j=1,\dots,m$, 
say $\{ \sigma_i^j \}_{i=1}^{B(m,j)}$ is the collection of $j$-faces of $\sigma$.
We will prove this lemma by induction on the dimensions of the faces of $\sigma$.
In particular, we will use the Fubini theorem to bound the integral of the ``slope'' of $\Phi^k$
by a bound on the integral of the ``slope'' of $\Phi^{k-1}$.
This will allow us to bound the integral of $g$ 
(which is the ``slope'' of $\Phi^m = F$).

We begin with the $(n+1)$-faces of $\sigma$.
Suppose $x \in \sigma_i^{n+1} \setminus \{c_i^{n+1} \}$ for some $i \in \{1,\dots,B(m,n+1)\}$.
If $x \notin \partial \sigma_i^{n+1}$,
then for any $y \in \Sigma_{\sigma}^{(n+1)}$ close enough to $x$,
in fact $y \in \sigma_i^{n+1}$
and $P^{n+1}(x)$ and $P^{n+1}(y)$ lie in the same $n$-face of $\sigma_i^{n+1}$.
In this case \eqref{toplip2} gives
$$
\frac{d(\tilde{f}(P^{n+1}(x)), \tilde{f}(P^{n+1}(y)))}{|x-y|} 
\leq L\tilde{C} \, \frac{|P_i^{n+1}(x) - P_i^{n+1}(y)|}{|x-y|}
\leq \nu L \tilde{C} \, \frac{\diam(\sigma)}{|x-c_i^{n+1}|}.
$$
for $y \in \Sigma_{\sigma}^{(n+1)}$ close enough to $x$.
Since each $\partial \sigma_i^{n+1}$ has $\mathcal{H}^{n+1}$ measure zero,
\begin{align*}
\int_{\Sigma_{\sigma}^{(n+1)}} 
	&\limsup_{y \to x , \, y \in \Sigma_{\sigma}^{(n+1)}} 
	\frac{d(\Phi^{n+1}(x), \Phi^{n+1}(y))^p}{|x-y|^p} \, d\mathcal{H}^{n+1}(x) \\
&= \sum_{i=1}^{B(m,n+1)} \int_{\sigma_i^{n+1} \setminus \partial \sigma_i^{n+1}} 
	\limsup_{y \to x , \, y \in \sigma_i^{n+1}} 
	\frac{d(\tilde{f}(P^{n+1}(x)),\tilde{f}(P^{n+1}(y)))^p}{|x-y|^p} \, d\mathcal{H}^{n+1}(x) \\
&\leq (\nu L \tilde{C})^p \sum_{i=1}^{B(m,n+1)} \int_{\sigma_i^{n+1}} 
	\frac{\diam(\sigma)^p}{|x-c_i^{n+1}|^p} \, d\mathcal{H}^{n+1}(x).
\end{align*}
In what follows, the constant $C$ may change value between lines in the inequalities
but will depend only on $m$, $n$, $p$, and $\gamma$.
We first estimate the integral over each $(n+1)$-face of $\sigma$.
Since $p<n+1$, we have
$$
\int_{\sigma_i^{n+1}} \, \frac{1}{|x-c_i^{n+1}|^p} \, d\mathcal{H}^{n+1}(x)
\leq C \, \mathcal{H}^{n+1}\left(\sigma_i^{n+1} \right)^{1-\frac{p}{n+1}} 
\leq C \, \diam(\sigma)^{n+1-p}.
$$
Therefore, on the entire $(n+1)$-skeleton, we have
$$
\int_{\Sigma_{\sigma}^{(n+1)}} \limsup_{y \to x , \, y \in \Sigma_{\sigma}^{(n+1)}} 
\frac{d(\Phi^{n+1}(x),\Phi^{n+1}(y))^p}{|x-y|^p} \, d\mathcal{H}^{n+1}(x) 
\leq L^p \, C\, \diam(\sigma)^{n+1}.
$$

Now suppose $k \in \{ n+1,\dots,m-1\}$ satisfies the following
for a constant $C$ depending only on $m$, $n$, $p$, and $\gamma$:
$$
\int_{\Sigma_{\sigma}^{(k)}} \limsup_{y \to x , \, y \in \Sigma_{\sigma}^{(k)}} 
\frac{d(\Phi^{k}(x),\Phi^{k}(y))^p}{|x-y|^p} \, d\mathcal{H}^{k}(x) 
\leq L^p \, C \, \diam(\sigma)^{k}.
$$
We have as before
\begin{align*}
&\int_{\Sigma_{\sigma}^{(k+1)}} \limsup_{y \to x , \, y \in \Sigma_{\sigma}^{(k+1)}} \frac{d(\Phi^{k+1}(x),\Phi^{k+1}(y))^p}{|x-y|^p} \, d\mathcal{H}^{k+1}(x)  \\
&\leq \sum_{i=1}^{B(m,k+1)} \int_{\sigma_i^{k+1} \setminus \partial \sigma_i^{k+1}}
\limsup_{y \to x , \, y \in \sigma_i^{k+1} } \frac{d(\Phi^{k} (P^{k+1}(x)),\Phi^{k} (P^{k+1}(y)))^p}{|P^{k+1}(x) - P^{k+1}(y)|^p} 
\frac{|P^{k+1}(x) - P^{k+1}(y)|^p}{|x-y|^p} \, d\mathcal{H}^{k+1}(x). \\
\end{align*}
Fix $i \in \{ 1, \dots, B(m,k+1) \}$. 
As before, we estimate the integral over each $(k+1)$-face of $\sigma$.
Without loss of generality (after a translation), we may assume $\sigma_i^{k+1}$ is centered at the origin.
We thus have by \eqref{toplip2}
\begin{align*}
&\int_{\sigma_i^{k+1} \setminus \partial \sigma_i^{k+1}} \limsup_{y \to x , \, y \in \sigma_i^{k+1}} 
\frac{d(\Phi^{k} (P^{k+1}(x)),\Phi^{k} (P^{k+1}(y)))^p}{|P^{k+1}(x) - P^{k+1}(y)|^p} 
	\frac{|P^{k+1}(x) - P^{k+1}(y)|^p}{|x-y|^p} \, d\mathcal{H}^{k+1}(x) \\
&\leq \nu^p \int_{\sigma_i^{k+1}} \limsup_{y \to x , \, y \in \sigma_i^{k+1}} \frac{d(\Phi^{k} (P^{k+1}(x)),\Phi^{k} (P^{k+1}(y)))^p}{|P^{k+1}(x) - P^{k+1}(y)|^p} \, 
	\frac{\diam(\sigma)^p}{|x|^p} \, d\mathcal{H}^{k+1}(x) \\
&\leq \nu^p \sum_{q =1}^{k+2} \int_{(P_i^{k+1})^{-1} (\sigma_{i_{q}}^{k})} 
	\limsup_{y \to x , \, y \in \sigma_i^{k+1}} \frac{d(\Phi^{k} (P^{k+1}(x)),\Phi^{k} (P^{k+1}(y)))^p}{|P^{k+1}(x) - P^{k+1}(y)|^p} \, \frac{\diam(\sigma)^p}{|x|^p} \, d\mathcal{H}^{k+1}(x) \\
\end{align*}
where $\sigma_{i_{1}}^{k}, \dots , \sigma_{i_{k+2}}^{k}$ are the $k$-dimensional faces of $\sigma_{i}^{k+1}$.
We will compute the integral of each summand in the last line.
Fix $q \in \{ 1, \dots, k+2 \}$.
The integral is invariant up to rotation, so we may assume without loss of generality that 
$\sigma_{i_{q}}^{k}$ is contained in the $k$-plane $\{b\} \times \mathbb{R}^k$.
Thus we may consider $(P_i^{k+1})^{-1} (\sigma_{i_{q}}^{k})$ a subset of $\mathbb{R}^{k+1}$.
Write $\hat{\sigma}_{i_q}^k = \{ \hat{z} \, | \, (b,\hat{z}) \in \sigma_{i_q}^k \} \subset \mathbb{R}^k$
so that 
$(P_i^{k+1})^{-1} (\sigma_{i_{q}}^{k}) 
= \{ (t,\hat{x}) \, | \, \hat{x} \in \frac{t}{b} \hat{\sigma}_{i_{q}}^{k}, \, t \in (0,b] \}$.
Thus since 
$$
\limsup_{y \to x , \, y \in \sigma_i^{k+1}} \frac{d(\Phi^{k} (P^{k+1}(x)),\Phi^{k} (P^{k+1}(y)))}{|P^{k+1}(x) - P^{k+1}(y)|}
\leq \limsup_{z \to P^{k+1}(x), \, z \in \sigma_{i_q}^k} \frac{d(\Phi^{k} (P^{k+1}(x)),\Phi^{k} (z))}{|P^{k+1}(x) - z|}
$$
for any $x \in (P_i^{k+1})^{-1} (\sigma_{i_q}^k \setminus \partial \sigma_{i_q}^k)$
and since $(P_i^{k+1})^{-1} (\partial \sigma_{i_q}^k)$ has $\mathcal{H}^{k+1}$ measure zero,
we have
\begin{align*}
\int_{(P_i^{k+1})^{-1} (\sigma_{i_{q}}^{k})} & \limsup_{y \to x , \, y \in \sigma_i^{k+1}} 
	\frac{d(\Phi^{k} (P^{k+1}(x)),\Phi^{k} (P^{k+1}(y)))^p}{|P^{k+1}(x) - P^{k+1}(y)|^p} \, 
	\frac{\diam(\sigma)^p}{|x|^p} \, d\mathcal{H}^{k+1}(x) \\
&\leq \int_0^b \int_{\frac{t}{b} \hat{\sigma}_{i_{q}}^{k}} 
	\limsup_{\hat{z} \to \frac{b}{t} \hat{x} , \, \hat{z} \in \hat{\sigma}_{i_q}^k} 
	\frac{d(\Phi^{k} (b,\frac{b}{t} \hat{x}),\Phi^{k} (b,\hat{z}))^p}{|(b,\frac{b}{t} \hat{x}) - (b,\hat{z})|^p} \, 
	\frac{\diam(\sigma)^p}{|(t,\hat{x})|^p} \, d\mathcal{H}^{k}(\hat{x}) \, dt \\
&= \int_0^b \int_{\hat{\sigma}_{i_{q}}^{k}} \left( \frac{t}{b} \right)^k 
	\limsup_{\hat{z} \to \hat{x}  , \, \hat{z} \in \hat{\sigma}_{i_q}^k} 
	\frac{d(\Phi^{k} (b,\hat{x}),\Phi^{k} (b,\hat{z} ))^p}{|(b,\hat{x}) - (b,\hat{z})|^p} \, 
	\frac{\diam(\sigma)^p}{\left( \frac{t}{b} \right)^p |(b, \hat{x})|^p} \, d\mathcal{H}^{k}(\hat{x}) \, dt \\
&\leq \int_0^b \left( \frac{t}{b} \right)^{k-p} \, dt \, \int_{\sigma_{i_{q}}^{k}} 
	\limsup_{z \to x  , \, z \in \sigma_{i_q}^k} 
	\frac{d(\Phi^{k} (x),\Phi^{k} (z))^p}{|x-z|^p} 
	\frac{\diam(\sigma)^p}{b^p} \, d\mathcal{H}^{k}(x) \\
&\leq \left( \frac{\diam(\sigma)}{b} \right)^p \, b \, L^p \, C \, \diam(\sigma)^{k}
\end{align*}
since $k-p>0$.
Since $b \geq \beta_{\sigma}$ and $b \leq \diam(\sigma)$, 
we may use \eqref{flat} to conclude
on the $(k+1)$-skeleton of $\sigma$
$$
\int_{\Sigma_{\sigma}^{(k+1)}} \limsup_{y \to x , \, y \in \Sigma_{\sigma}^{(k+1)}} 
	\frac{d(\Phi^{k+1}(x),\Phi^{k+1}(y))^p}{|x-y|^p} \, d\mathcal{H}^{k+1}(x) 
\leq L^p \, C \, \diam(\sigma)^{k+1}.
$$

By way of induction, then, we have
$$
\int_{\sigma} g(x)^p \, d \mathcal{H}^m(x) 
= \int_{\sigma \setminus \partial \sigma} \limsup_{y \to x} 
	\frac{d(\Phi^{m}(x),\Phi^{m}(y))^p}{|x-y|^p} \, d\mathcal{H}^{m}(x) 
\leq L^p \, C \, \diam(\sigma)^{m}
$$
since $\Sigma_{\sigma}^{(m)} = \sigma$ and $\Phi^m = F$ on $\sigma$.
Therefore, we have
\begin{align*}
\int_{\Omega \setminus Z} g(x)^p \, dx
\leq \sum_{i=1}^{\infty} \int_{\sigma_i} g(x)^p \, dx
\leq L^p \, C \, \sum_{i=1}^{\infty} \diam(\sigma_i)^{m}.
\end{align*}
The number of $m$-simplices in each cube in the Whitney decomposition of $\mathbb{R}^m \setminus Z$ 
is bounded by a constant $C$ depending only on $m$.
Hence
\begin{align*}
\sum_{i=1}^{\infty} \diam(\sigma_i)^{m} 
= \sum_{Q} \sum_{\sigma \subset Q} \diam(\sigma)^{m}
\leq \sum_{Q} \sum_{\sigma \subset Q} \diam(Q)^{m}
&\leq C \sum_{Q} \diam(Q)^{m} \\
&\leq C \mathcal{H}^m(\Sigma^{(m)})
\end{align*}
where these sums are taken over all cubes $Q$ in the Whitney decomposition that meet $\Omega$.
Notice that, for any $x,y \in \Sigma^{(m)}$ and cubes $Q_x$ and $Q_y$ containing them, 
we have
\begin{align*}
|x-y| 
&\leq \diam(Q_x) + d(Q_x,Q_y) + \diam(Q_y) \\
&\leq d(Q_x,Z) + d(Q_x,Q_y) + d(Q_y,Z)
\leq 3 \diam(\Omega).
\end{align*}
Therefore,
$
\mathcal{H}^m(\Sigma^{(m)}) \leq C \, \diam(\Sigma^{(m)})^m \leq C (\diam(\Omega))^m
$,
and so
$
\Vert g\Vert _{L^p(\Omega \setminus Z)} \leq CL (\diam(\Omega))^{m/p}
$
for a constant $C>0$ depending only on $m$, $n$, $p$, and the Lipschitz connectivity constant $\gamma$ of $Y$.
In particular, $g \in L^p(\Omega \setminus Z)$.
\end{proof}

\end{document}